\documentclass[a4paper,reqno,12pt]{amsart}
\usepackage{verbatim}
\usepackage{amssymb,amsmath,amsthm}
\usepackage{amsfonts}
\usepackage{array}
\usepackage{xcolor}
\usepackage{youngtab}

\newtheorem{theorem}{Theorem}[section]
\newtheorem{lemma}{Lemma}[section]
\newtheorem{corollary}{Corollary}[section]

\newtheorem{proposition}{Proposition}[section]
\newtheorem{example}{Example}[section]

\newtheorem{definition}[theorem]{Definition}

\newtheoremstyle{remark}
    {\dimexpr\topsep/2\relax} 
    {\dimexpr\topsep/2\relax} 
    {}          
    {}          
    {\bfseries} 
    {.}         
    {.5em}      
    {}          

\theoremstyle{remark}
\newtheorem{remark}{Remark}[section]

\makeatletter
\@namedef{subjclassname@2020}{%
  \textup{2020} Mathematics Subject Classification}
\makeatother

\begin{document}

\title[]
{Further study on MacMahon-type sums of divisors }

\author{Tewodros Amdeberhan} 

\address{Department of Mathematics,
Tulane University, New Orleans, LA 70118, USA}
\email{tamdeber@tulane.edu}

\author{George E. Andrews}\thanks{The second author is partially supported by Simon Foundation Grant 633284.}

\address{Department of Mathematics, Penn State University, University Park, PA 16802, USA}
 \email{gea1@psu.edu} 

\author{Roberto Tauraso}

\address{Dipartimento di Matematica, 
Università di Roma ``Tor Vergata'', 00133 Roma, Italy}
\email{tauraso@mat.uniroma2.it}

\subjclass[2020]{Primary 11M32, 11P83; Secondary 11F37.}

\keywords{partitions, divisor functions, Chebychev polynomials, congruences.}

\begin{abstract} This paper is devoted to the study of
\begin{align*} 
U_t(a,q):=\sum_{1\leq n_1<n_2<\cdots<n_t}\frac{q^{n_1+n_2+\cdots+n_t}}{(1+aq^{n_1}+q^{2n_1})(1+aq^{n_2}+q^{2n_2})\cdots(1+aq^{n_t}+q^{2n_t})}
\end{align*}
when $a$ is one of $0, \pm 1, \pm2$. The idea builds on our previous treatment of the case $a=-2$. It is shown that all these functions lie in the ring of quasi-modular forms. Among the more surprising findings is
$$U_2(1,q)=\sum_{n\geq1} \frac{q^{3n}}{(1-q^{3n})^2}.$$
\end{abstract} 

\maketitle

\section{Introduction}

\smallskip
\noindent
The object of this paper is the study of 
\begin{align} \label{defofU} 
U_t(a,q):=\sum_{1\leq n_1<n_2<\cdots<n_t}\frac{q^{n_1+n_2+\cdots+n_t}}{\prod_{k=1}^t(1+aq^{n_k}+q^{2n_k})}=\sum_{n\geq0} MO(a,t;n)q^n,
\end{align}
where $a$ is among $0, \pm1, \pm2$. As will become clear, the $U_t(a;q)$ are most interesting when the roots of 
$$1+ax+x^2=0$$
are also roots of unity, and this occurs precisely for $a=0, \pm1, \pm2$.

\smallskip
\noindent
We shall prove that these functions are  quasi-modular forms. Thus we can expect  interesting arithmetic consequences some of which we describe in detail. In Section 2, we prove a necessary theorem connecting $U_t(a,q)$ with modified Chebychev polynomials and prove that $MO(1,t;3n+2)=0$. In Sections 3, we prove analytically the following result:

\begin{theorem} \label{ThmU2(1,q)} We have that
\begin{align} \label{U2(1,q)}
U_2(1,q)=\sum_{n\geq1} \frac{q^{3n}}{(1-q^{3n})^2}.
\end{align}
\end{theorem}

\noindent
In Section 4, we give a combinatorial proof of the result in \eqref{U2(1,q)}. Section 5 provides alternative forms of $U_t(a,q)$ necessary to discerning the quasi-modular nature of these functions. Section 6 treats quasi-modularity in the cases $a=\pm2$. Section 7 explores the cases $t=1$, paving the way for Section 8 where we reveal quasi-modularity of the remaining cases $a=0, \pm1$. The paper concludes with two appendices, the first uses the Wilf-Zeilberger (WZ) method to treat some of the binomial coefficient summations needed in this paper, and the second utilizes Riordan arrays for the same objective.

\section{The Chebychev connection}

\noindent
In this section, we begin by recalling the Andrews-Rose identity \cite{Andrews-Rose} which involves the \emph{Chebychev polynomials of the first kind} defined by 
\begin{align*} 
T_n(\cos\theta)=\cos(n\theta).
\end{align*}
In \cite[Theorem 1]{Andrews-Rose}, it is proved that 
\begin{align} \label{ChebyAR}
2\sum_{n\geq0}T_{2n+1}\left(\frac{x}2\right)q^{n^2+n}&=x(q^2;q^2)^3_{\infty} \prod_{n\geq1}\left(1+\frac{x^2q^{2n}}{(1-q^{2n})^2}\right) \\ \nonumber
&=(q^2;q^2)_{\infty}^3 \sum_{t\geq0} U_t(-2,q)x^{2t+1},
\end{align}
where $U_t(a,q)$ is defined in \eqref{defofU}, and 
$$(A;q)_n=(1-A)(1-Aq)\cdots (1-Aq^{n-1}) \qquad \text{for $0\leq n\leq\infty$}.$$

\noindent
We note that $T_n(x)$ is alternately an even and odd function, depending on $n$. To make \eqref{ChebyAR} more manageable for our purposes, we write
\begin{align} \label{ChebyAlt}
to_n(x)=\frac{T_{2n+1}(\sqrt{x})}{\sqrt{x}}.
\end{align}
Thus equation \eqref{ChebyAR} now becomes
$$\sum_{n\geq 0}to_n\left(\frac{x}{4}\right)q^{\binom{n+1}{2}}
=(q;q)^3_{\infty}\prod_{n=1}^{\infty}\left(1+\frac{xq^n}{(1-q^n)^2}\right)
=(q;q)_{\infty}^3 \sum_{t\geq0} U_t(-2,q)x^t.$$
Consequently, we obtain the following.

\begin{theorem} We have
\begin{align}\label{id2} \sum_{t\geq0} U_t(a,q)x^t
=&\prod_{n=1}^{\infty}\frac{1}{(1+aq^n+q^{2n})(1-q^n)}
\cdot \sum_{n\geq 0}to_n\left(\frac{x+a+2}{4}\right)q^{\binom{n+1}{2}}.
\end{align}
\end{theorem}
\begin{proof} We proceed as follows,
\begin{align*} \sum_{t\geq0} U_t(a,q)x^t
=&\prod_{n=1}^{\infty}\left(1+\frac{xq^n}{1+aq^n+q^{2n}}\right) \\
=&\prod_{n=1}^{\infty}\frac{1}{1+aq^n+q^{2n}}
\cdot \prod_{n=1}^{\infty}\left(1+(x+a)q^n+q^{2n}\right) \\
=&\prod_{n=1}^{\infty}\frac{(1-q^n)^2}{1+aq^n+q^{2n}}
\cdot \prod_{n=1}^{\infty}\left(1+\frac{(x+a+2)q^n}{(1-q^n)^2}\right) \\
=&\prod_{n=1}^{\infty}\frac{1}{(1+aq^n+q^{2n})(1-q^n)}
\cdot \sum_{n\geq 0}to_n\left(\frac{x+a+2}{4}\right)q^{\binom{n+1}{2}}. \qedhere
\end{align*}
\end{proof}

\noindent
This core identity allows us an immediate congruence.

\begin{theorem} \label{MO3n+2} We have $MO(1,t;3n+2)=0$.
\end{theorem}
\begin{proof} By \eqref{id2}, 
\begin{align*} \sum_{t\geq0} U_t(1,q)x^t &= \prod_{n\geq1}\frac1{(1+q+q^{2n})(1-q^n)} \sum_{n\geq0} to_n\left(\frac{x+3}4\right)q^{\binom{n+1}2} \\
&=\frac1{(q^3;q^3)_{\infty}}\sum_{n\geq0}to_n\left(\frac{x+3}4\right)q^{\binom{n+1}2}.
\end{align*}
Now, $\binom{n+1}2$ is never congruent to $2$ modulo $3$. Hence there are no powers of $q$ in $U_t(1,q)$ congruent to $2$ modulo $3$. This establishes our theorem.
\end{proof}

\section{Analytic proof of Theorem~\ref{ThmU2(1,q)}}

\noindent
We start with the following preliminary lemma.

\begin{lemma} \label{Lm1} Let $\omega(n):=\frac{n(3n+1)}2$. Then, 
$$\sum_{n=1}^{\infty}\frac{q^{3n}}{(1-q^{3n})^2}
=\frac{1}{(q^3;q^3)_{\infty}} \cdot \sum_{n\in\mathbb{Z}} (-1)^{n-1} \omega(n)\,q^{3\omega(n)}.$$
\end{lemma}
\begin{proof} Consider
\begin{align*}
q\frac{d}{dq} (q^3;q^3)_{\infty}
&=-q\cdot (q^3;q^3)_{\infty} \cdot \sum_{n=1}^{\infty}\frac{3nq^{3n-1}}{1-q^{3n}} \\
&=-3\cdot (q^3;q^3)_{\infty} \cdot \sum_{n=1}^{\infty}\sum_{m=1}^{\infty}nq^{3nm} \\
&=-3\cdot (q^3;q^3)_{\infty} \cdot \sum_{m=1}^{\infty}\frac{q^{3m}}{(1-q^{3m})^2}.
\end{align*}
On the other hand, the identity $(q;q)_{\infty}=\sum_{\mathbb{Z}}(-1)^nq^{\omega(n)}$ implies that
\begin{align*}
q\frac{d}{dq} (q^3;q^3)_{\infty}&=q\frac{d}{dq}\sum_{n=-\infty}^{\infty} (-1)^n q^{3\omega(n)}
=3\sum_{n=-\infty}^{\infty}(-1)^n\omega(n)q^{3\omega(n)}.
\end{align*}
The conclusion is immediate from here. 
\end{proof}

\subsection{Proof of Theorem~\ref{ThmU2(1,q)}.} 
Let us now recall \cite{AbrStegun} one of the representations of $T_n(x)$ adjusted to what we defined \eqref{ChebyAlt} as $to_n(x)$. Namely,
\begin{align} \label{Chebysum}
to_n(x)=\frac{T_{2n+1}(\sqrt{x})}{\sqrt{x}}=(2n+1)\sum_{k=0}^n(-1)^{n+k}\binom{n+k+1}{2k+1}\frac{(4x)^k}{n+k+1}.
\end{align}

\noindent
Next we note that $U_2(1,q)$ is the coefficient of $x^2$ in the expansion (see Theorem~\ref{id2})
\begin{align*} 
\sum_{t\geq0} U_t(1,q)x^t=\prod_{n\geq1}\left(1+\frac{xq^n}{1+q^n+q^{2n}}\right)
=\frac1{(q^3;q^3)_{\infty}} \cdot \sum_{n\geq0} to_n\left(\frac{x+3}4\right)q^{\binom{n+1}2}.
\end{align*}
Based on the expression
\begin{align} \label{StandardTn}
to_n\left(\frac{x+3}{4}\right)&=(2n+1)\sum_{k=0}^n (-1)^{n+k} \binom{n+k+1}{2k+1}\frac{(x+3)^k}{n+k+1}  \\ \nonumber
&=(2n+1)\sum_{k=0}^n (-1)^{n+k} \frac{\binom{n+k+1}{2k+1}}{n+k+1}\sum_{i=0}^k\binom{k}{i}3^{k-i}x^i,
\end{align}
we want the coefficient of $x^2$ in \eqref{StandardTn}, which is
$$c_n:=(2n+1)\sum_{k=0}^n (-1)^{n+k} \frac{\binom{n+k+1}{2k+1}}{n+k+1}\binom{k}{2}3^{k-2}.$$
We claim that
$$c_n=\begin{cases}
(-1)^{j-1}\frac{j(3j+1)}{2}&\text{if $n=3j$,}\\
0&\text{if $n=3j+1$,}\\
(-1)^{j-1}\frac{j(3j-1)}{2}&\text{if $n=3j-1$.}
\end{cases}$$
This is a standard binomial coefficient identity that can easily be verified (see Example~\ref{WZ2} in Appendix 1 or Example~\ref{R2} in Appendix 2).
That means,
$$U_2(1,q)=\frac1{(q^3;q^3)_{\infty}}\cdot \left(\sum_{j\geq0}(-1)^{j-1}\frac{j(3j+1)}2q^{\binom{3j+1}2}
+\sum_{j\geq0}(-1)^{j-1}\frac{j(3j-1)}2q^{\binom{3j}2}\right).$$
The proof is completed after comparing this formula with that of Lemma~\ref{Lm1}.

\bigskip
\noindent
As a companion for Theorem~\ref{MO3n+2}, we are now in a position to prove the next result.

\begin{theorem}  \label{MO3n+1}  The following congruence holds true:
$$MO(1,3;3n+1) \equiv 0 \pmod 3.$$
\end{theorem}

\begin{proof} We will borrow a result for $U_3(1,q)=\sum_{n\geq0}MO(1;3,n)q^n$ from Section 5:
\begin{align*}
U_3(1,q)&= \frac1{(q^3;q^3)_{\infty}} 
\sum_{n\geq 0} q^{\binom{n+1}2}\sum_{k=0}^{n} \frac{(-1)^{n+k} (2n+1)\binom{n+k+1}{2k+1}}{n+k+1}\binom{k}{3}  3^{k-3}\\
&= \frac{1}{(q^3;q^3)_{\infty}} 
\sum_{n\geq 0} q^{\binom{n+1}2}\sum_{k=3}^n (-1)^{n+k} \left(\binom{n+k+1}{2k+1}+\binom{n+k}{2k+1}\right)\binom{k}{3} 3^{k-3}.
\end{align*}
Since we are interested in the coefficients of $q^{3n+1}$ in  $U_3(1,q)$, and all the powers of $q$ in $\frac1{(q^3;q^3)_{\infty}}$ are multiple of $3$, it suffices to check the case when
$\binom{n+1}2\equiv 1\pmod3$, that is $n=3m+1$, and show that the inner sum
$$\sum_{k=3}^{3m+1} (-1)^{3m+k+1}\left(\binom{3m+k+2}{2k+1}+\binom{3m+k+1}{2k+1}\right)\binom{k}{3}  3^{k-3}$$
is divisible by $3$.
Actually, each summands is immediately divisible by $3$ due to the term $3^{k-3}$ unless (perhaps) when $k=3$. 
Replacing this value to compute the corresponding term, we obtain
\begin{align*}
(-1)^{3m}\left(\binom{3m+5}{7}+\binom{3m+4}{7}\right)\binom{3}{3}  3^{0}&= (-1)^m\frac{3(2m+1)}{7}\binom{3m+4}{6} \\
& \equiv 0\pmod 3.
\end{align*}
This completes the proof.
\end{proof}

\section{A combinatorial proof of Theorem~\ref{ThmU2(1,q)}}

\smallskip
\noindent
Let us recall this alternative representations of $U_2(1,q)$,
\begin{align*} 
U_2(1,q)&=\sum_{1\leq n_1<n_2}\frac{q^{n_1+n_2}}{(1+q^{n_1}+q^{2n_1})(1+q^{n_2}+q^{2n_2})}\\
&=\sum_{1\leq n_1<n_2}\frac{q^{n_1+n_2}(1-q^{n_1})(1-q^{n_2})}{(1-q^{3n_1})(1-q^{3n_2})}\\
&=\sum_{n\geq 1}\frac{q^{3n}}{(1-q^{3n})^2}.
\end{align*}
\noindent
We use the following notation
$$n_1^{f_1}n_2^{f_2}$$
to denote the partition of $f_1n_1+f_2 n_2$ into different parts $n_1$ and $n_2$ wherein $n_1$ appears $f_1$ times and $n_2$ appears $f_2$ times.

\smallskip
\noindent
Let $P_0(n)$ denotes the number of partitions of $n$ involving two different parts (each may occur any number of time)
$$n_1^{f_1}n_2^{f_2}$$
where neither $f_1$ nor $f_2$ is divisible by $3$, and $f_1\equiv f_2\pmod{3}$.

\smallskip
\noindent
Let $P_1(n)$ denotes the number of partitions of $n$ involving two different parts
$$n_1^{f_1}n_2^{f_2}$$
where neither $f_1$ nor $f_2$ is divisible by $3$, and $f_1\not\equiv f_2\pmod{3}$.

\smallskip
\noindent
Thus we may reformulate Theorem~\ref{ThmU2(1,q)} as
\begin{align} \label{PoP1}
P_0(n)-P_1(n)=\begin{cases} 0 
&\text{if $n\not \equiv 0 \pmod{3}$,}\\
\sigma\big(\frac{n}{3}\big) &\text{if $n\equiv 0 \pmod{3}$.}
\end{cases}
\end{align}
The analytic form of \eqref{PoP1} is clearly
$$\sum_{n\geq 1}(P_0(n)-P_1(n))q^n=\sum_{n\geq 1}\frac{q^{3n}}{(1-q^{3n})^2}.$$
To prove this, we begin with a proposed bijection between the partitions enumerated by $P_0(n)$ and those enumerated by $P_1(n)$.

\smallskip
\noindent
First we map $P_1(n)$ partitions into $P_2(n)$ partitions. We begin with the partitions
$$n_1^{f_1}n_2^{f_2}$$
where $f_1\not\equiv f_2\pmod{3}$ and $3$ divides neither $f_1$ nor $f_2$.
Without loss of generality we take $f_2>f_1$ (equality is impossible because $f_1\not\equiv f_2\pmod{3}$)
$$n_1^{f_1}n_2^{f_2}\mapsto (n_1+n_2)^{f_1}n_2^{f_2-f_1}.$$
Clearly the image is a $P_2$ partition because $f_2-f_1\not \equiv f_2 \pmod{3}$ and thus must be congruent to $f_1$ (keep in mind there are only $2$ non-zero residue classes modulo $3$).

\smallskip
\noindent
This is evidently an injection of the $P_1$ partitions into $P_0$ partitions.

\smallskip
\noindent
For the reverse mapping we assume $n_1>n_2$ (and $f_1\not \equiv f_2 \pmod{3}$, neither $\equiv 0 \pmod{3}$)
$$n_1^{f_1}n_2^{f_2}\mapsto (n_1-n_2)^{f_1}n_2^{f_1+f_2}.$$
The reverse image clearly has $f_1\not\equiv f_1+f_2 \pmod{3}$ (and neither $f_1$ nor $f_1+f_2$ is congruent to $0$ modulo $3$). Also both $n_1-n_2$ and $n_2$ are positive integers. However, the image is not a $P_0$ partition precisely when $n_1=2n_2$ because then there are not two different parts.

\smallskip
\noindent
So let us examine partitions of the form
$$(2d)^{f_1}d^{f_2}$$
with $f_1\equiv f_2 \pmod{3}$. The number being partitioned is
$$f_1(2d)+f_2 d=d(2f_1+f_2)=n.$$
Now $2f_1+f_2\equiv -f_1+f_2\equiv 0 \pmod{3}$; so $3\mid n$ and $d$ must be a divisor of $n/3$.

\smallskip
\noindent
Thus if $n\not\equiv 0\pmod{3}$, we have a bijection and $P_1(n)=P_0(n)$. Finally we may take $n=3\nu$. In how many ways can we solve
$$f_1(2d)+f_2 d=3\nu,$$
or equivalently
$$2f_1+f_2=3\frac{\nu}{d}?$$
This is solvable for $f_1$ provided $3\frac{\nu}{d}-f_2$ is even. Thus $f_2$ may be chosen from
$$\text{$1$ or $2$, $4$ or $5$, $7$ or $8$, $\dots$, $3\frac{\nu}{d}-2$ or $3\frac{\nu}{d}-1$}.$$
Therefore, there are $\frac{\nu}{d}$ choices possible for $f_2$ (and $f_1$ is uniquely determined once $f_2$ is chosen).
Hence for each divisor $d$ of $\nu$ (recall $n=3\nu$), there are $\frac{\nu}{d}$ partitions in $P_0(n)$  without an image in $P_1(n)$. Hence
$$P_0(3\nu)-P_1(3\nu)=\sum_{d\mid \nu}\frac{\nu}{d}=\sum_{d\mid \nu}d=\sigma(\nu)=\sigma\Big(\frac{n}{3}\Big)$$
and Theorem~\ref{ThmU2(1,q)} is proved.

\begin{example} For $n=9=3\cdot 3$, we obtain
\begin{align*}
&711\mapsto 81\\
&522\mapsto 72\\
&441\mapsto 54\\
&411111\mapsto 51111, 63, 22221, 2211111, 21111111.
\end{align*}
\end{example}

\section{A combined setup for all $U_t(a,q)$}

\noindent
By adopting the convention $U_0(a,q):=1$, let us write a generating series for the family of functions $U_t(a,q)$ as 
$$F(x;a,q):=\sum_{t\geq0} U_t(a,q)\, x^t.$$
Also recall \eqref{id2},
$$F(x;a,q)=\prod_{m\geq1}\frac1{(1+aq^m+q^{2m})(1-q^m)}\cdot \sum_{n\geq0}to_n\left(\frac{x+a+2}4\right)q^{\binom{n+1}2}.$$
As a result, we may extract the coefficient, denoted $\mathbf{[x^t]}$, of $x^t$ from both sides:
$$U_t(a,q)= \prod_{m\geq1}\frac1{(1+aq^m+q^{2m})(1-q^m)}\cdot \sum_{n\geq0}q^{\binom{n+1}2}\mathbf{ [x^t]} \, \,to_n\left(\frac{x+a+2}4\right).$$
From \eqref{StandardTn}, we know that
\begin{align} \label{coeffCheby} \mathbf{ [x^t]} \, \,to_n\left(\frac{x+a+2}4\right)
=(2n+1)\sum_{k=0}^n(-1)^{n+k} \frac{\binom{n+k+1}{2k+1}}{n+k+1}\binom{k}t (a+2)^{k-t}. \end{align}

\noindent
The case $a=-2$ goes back to Andrews-Rose~\cite[Corollary 2]{Andrews-Rose}:
\begin{align} \label{Mac-2}
U_t(-2,q)= \frac1{(q;q)^3_{\infty}} \cdot \sum_{n\geq0} (-1)^{n+t} \frac{2n+1}{2t+1} \binom{n+t}{2t} q^{\binom{n+1}2} .
\end{align}

\noindent
For $a=2$, the coefficient of $x^t$ in \eqref{id2} is
$$c_n:=(2n+1)\sum_{k=0}^n (-1)^{n+k} \frac{\binom{n+k+1}{2k+1}}{n+k+1}\binom{k}{t}4^{k-t}=\binom{n+t}{2t}.$$
This is known as \emph{Moriarty identity}, see Gould's collection \cite[3.161]{Gould} or Appendix 2 or  Example~\ref{WZ1} in Appendix 1 below.
Therefore, we have
\begin{align} \label{Mac2}
U_t(2,q)=\frac{(q;q)_{\infty}}{(q^2;q^2)_{\infty}^2} \cdot
\sum_{n\geq 0}\binom{n+t}{2t}q^{\binom{n+1}{2}}.
\end{align}

\noindent
The case $a=-1$ is:
$$U_t(-1,q)= \frac{(q^2;q^2)_{\infty}(q^3;q^3)_{\infty}}{(q;q)_{\infty}^2(q^6;q^6)_{\infty}} \cdot
\sum_{n\geq 0} \sum_{k=0}^n \frac{(-1)^{n+k}(2n+1)\binom{n+k+1}{2k+1}}{n+k+1}\binom{k}t  q^{\binom{n+1}{2}}.$$
The case $a=1$ is:
$$U_t(1,q)= \frac1{(q^3;q^3)_{\infty}} \cdot
\sum_{n\geq 0} \sum_{k=0}^n \frac{(-1)^k(2n+1)\binom{n+k+1}{2k+1}}{n+k+1}\binom{k}t  3^{k-t} q^{\binom{n+1}{2}}.$$

\noindent
For $a=0$, by the computations made in the Appendix 2 below, we may write
\begin{align*}
U_t(0,q)=& \frac{(q^2;q^2)_{\infty}}{(q;q)_{\infty}(q^4;q^4)_{\infty}} \cdot
\sum_{n\geq 0}(-1)^{n+\lfloor (n+t)/2\rfloor} \binom{\lfloor (n+t)/2\rfloor}{t} q^{\binom{n+1}{2}}.
\end{align*}

\section{Quasimodular structure for $a=\pm2$}

\noindent We introduce the rational functions 
$$Q_m(a,q):=\frac{q^m}{1+aq^m+q^{2m}}$$
so that we may write the above generating function as
$$F(x;a,q)=\prod_{m\geq1}(1+Q_m(a,q) \,x)=\sum_{t\geq0} U_t(a,q)\, x^t.$$
One also has that
$$-\log \left(F(-x;a,q)\right) = \sum_{r\geq1} H_r(a,q) \frac{x^r}r \quad \text{where} \quad
H_r(a,q):=\sum_{m\geq1} Q_m^r(a,q).$$

\subsection{The case $a=-2$} Let's revive the umbral expansion \cite[Example 7.1]{AAT} 
\begin{align} \label{umbral1}
(2r-1)!\, H_r(-2,q)=\mathbf{S}(\mathbf{S}^2-1^2)(\mathbf{S}^2-2^2)\cdots (\mathbf{S}^2-(r-1)^2) 
\end{align}
where 
$$H_r(-2,q)=\sum_{k\geq1} \frac{q^{rk}}{(1-q^k)^{2r}} \qquad \text{and} \qquad \mathbf{S}_j(q):=\sum_{k\geq1}\frac{k^jq^k}{1-q^k}.$$

\smallskip
\noindent
From Andrews-Rose \cite{Andrews-Rose}, we know that the $U_t(-2,q)$ lie in the ring of quasimodular forms and hence freely generated by the Eisenstein series $E_2=24D(\eta)/\eta, E_4$ and $E_6$ where 
$\eta(\tau)=q^{1/24} \prod_{n\geq1}(1-q^n)$ denotes the \emph{Dedekind eta-function}. So, the derivative $\partial_{E_2}$ is meaningful here.

\begin{proposition} It is true that
$$\partial_{E_2} U_t(-2,q) = -\frac12\sum_{j=1}^t \frac{U_{t-j}(-2,q)}{j^2\binom{2j}j}.$$
\end{proposition}
\begin{proof}
Operating with $\partial_{E_2}$ on $-\log\left(F(-x;-2,q)\right) =\sum_{r\geq1}H_r(-2,q)\frac{x^r}r$ leads to
$$-\sum_t (-x)^t\partial_{E_2}U_t(-2,q)= F(-x;-2,q)\sum_r\frac{x^r}r\partial_{E_2}H_r(-2,q)$$ 
from which we obtain (by comparing powers of $x^t$)
\begin{align*}
\partial_{E_2}U_t(-2,q)
&=-\sum_{j=1}^t (-1)^j \frac{U_{t-j}(-2,q)}j\cdot \partial_{E_2}H_j(-2,q) \\
&=-\sum_{j=1}^t (-1)^j \frac{U_{t-j}(-2,q)}{j(2j-1)!}\cdot \partial_{E_2}\left[\mathbf{S}\prod_{\ell=1}^{j-1}(\mathbf{S}^2-\ell^2)\right] \\
&=-\sum_{j=1}^t (-1)^j \frac{U_{t-j}(-2,q)}{j(2j-1)!} \cdot (-1)^j\cdot\frac{1^22^2\cdots(j-1)^2}{24}
\end{align*}
where the two facts $E_2=1-24\,\mathbf{S}_1$ and the identity in \eqref{umbral1} have been utilized. So, we infer that
$$\partial_{E_2}U_t(-2,q)=-\frac1{12}\sum_{j=1}^t \frac{U_{t-j}(-2,q)}{j^2\binom{2j}j}. $$ 
\end{proof}

\subsection{The case $a=2$} 

\begin{lemma} \label{lemma1_in_2} Let $D:=q\frac{d}{dq}$. If  $A_t(q):=\sum_{n\geq0}(2n+1)^tq^{\binom{n+1}2}, B(q):=\sum_{n\geq0}q^{\binom{n+1}2}$ and
 $C(q):=q^{\frac18}\prod_{m\geq1}(1+q^m)(1-q^{2m})$ then we have the property that
$$8^tq^{-\frac18}\cdot D^t C(q)=(1+8D)^t B(q)=A_{2t}(q) \qquad \text{when $t\geq0$ is an integer}.$$
\end{lemma}
\begin{proof} First recall \emph{Jacobi's triple product} (see \cite[p. 35, Entry 19]{Berndt}) which may be stated in the manner
\begin{align} \label{Jacobi} (q;q)_{\infty}\,(-z^{-1};q)_{\infty}\, (-zq;q)_{\infty}=\sum_{n\in\mathbb{Z}}\, q^{\binom{n+1}2}\, z^n.
\end{align}
Choosing $z=1$ leads to a well-known identity
\begin{align} \label{Jacobi1}
\prod_{m\geq1}(1+q^m)(1-q^{2m})=\sum_{n\geq0}\, q^{\binom{n+1}2} \end{align}
which ensures validity of the case $t=0$ of this lemma. In the next step, apply  logarithmic differentiation in \eqref{Jacobi1} in tandem with the product rule for derivatives on $C(q)$. Then proceed with induction on $t$ to complete the proof.
\end{proof}

\begin{lemma} \label{lemma2_in_2} Let $T_t(q):=4^t\sum_{n\geq0}\frac{(n+t)!}{(n-t)!} q^{\binom{n+1}2}$. Preserving notations from Lemma~\ref{lemma1_in_2}, we have the umbral relation
$$T_t=A^0(A^2-1^2)(A^2-3^2)\cdots(A^2-(2t-1)^2).$$

\end{lemma}
\begin{proof} The statement boils down to the elementary fact that
$$4^t\cdot \frac{(n+t)!}{(n-t)!}=\prod_{\ell=1}^t((2n+1)^2-(2\ell-1)^2). \qedhere$$
\end{proof}

\begin{theorem} We have the differential-difference equation
$$U_t(2,q)=\frac1{t(2t-1)}\left(D+U_1(2,q)-\binom{t}2\right)U_{t-1}(2,q)$$
together with  $U_1(2,q)=\mathbf{S}_1(q)-4\,\mathbf{S}_1(q^2)=-\frac18-\frac1{24}E_2(q)+\frac16E_2(q^2)$.
\end{theorem}
\begin{proof} The definition of $T_t(q)$ (see Lemma~\ref{lemma2_in_2}), notations from Lemma~\ref{lemma1_in_2} and formula \eqref{Mac2} imply the relation
$$U_t(2,q)=\alpha_t\,\frac{T_t(q)}{B(q)} \qquad \text{where} \qquad \alpha_t:=\frac1{4^t(2t)!}.$$
We prove this theorem by induction on $t$.  The claim on $U_1(2,q)$ is obvious once we notice that $\mathbf{S}_1(q)=\sum_n\frac{q^n}{(1-q^n)^2}$. Therefore, we need to show that
\begin{align} \label{diff_eqn}
\frac{\alpha_{t-1}}8\, \frac{T_t(q)}{B(q)}
=\alpha_{t-1}D\left(\frac{T_{t-1}(q)}{B(q)}\right)+\alpha_{t-1}\,\frac{T_1(q)\, T_{t-1}(q)}{8B^2(q)}-\alpha_{t-1}\frac{t(t-1)}2\,\frac{T_{t-1}(q)}{B(q)}.
\end{align}
To this end, it suffices to observe that (using Lemma~\ref{lemma1_in_2} and Lemma~\ref{lemma2_in_2})
\begin{align*}
D\left(\frac{T_{t-1}}{B}\right)
&=\frac{B(DT_{t-1})-T_{t-1}(DB)}{B^2}=\frac{DT_{t-1}}B-\frac{T_{t-1}}{B^2}\left(\frac{A_2-B}8\right) \\
&=\frac{DT_{t-1}}B-\frac{T_{t-1}}{B^2}\left(\frac{A_2-A_0}8\right)=\frac{DT_{t-1}}B-\frac{T_{t-1}}{B^2}\frac{T_1}8.
\end{align*}
Thus, after some simplifications, equation \eqref{diff_eqn} amounts to $T_t=(8D-4t(t-1)) T_{t-1}$.
Using the definition of $T_t$ and the derivation $D$, this is equivalent to
$$\sum_{n\geq0}\frac{4^t(n+t)!}{(n-t)!}q^{\binom{n+1}2}=\sum_{n\geq0}\left( \frac{4^tn(n+1)(n+t-1)!}{(n-t+1)!} - \frac{4^tt(t-1) (n+t-1)!}{(n-t+1)!} \right)q^{\binom{n+1}2}.$$
Finally, we rearrange the factorials on the right side to reduce to the left side.
\end{proof}

\begin{corollary} \label{quasi2} The function $U_t(2,q)$ belongs to $\widetilde{M}_{\leq2t}(\Gamma_0(2))$, where  $\widetilde{M}_{\leq2t}(\Gamma_0(2))$ is the space of quasimodular forms of mixed weight of at most $2t$ on $\Gamma_0(2)$. 
\end{corollary}

\begin{proof} From Theorem~\ref{quasi2}, we know that $U_1(2,q)$ is quasimodular of level $2$, weight at most $2$. By the same theorem, and since these rings are stable under the operator action $D$, our assertion becomes rather evident.
\end{proof}

\section{The special cases $t=1$ and $a=0$ or $\pm1$ for $U_t(a,q)$}

\noindent
We know that $U_t(-2,q)$ is quasimodular by work of Andrews-Rose \cite{Andrews-Rose}, and $U_t(2,q)$ is quasimodular by Section 8 above.
In the present section, we would note the following $t=1$ cases for the other values of $a$.

\smallskip
\noindent
We still use $\omega(n)=\frac{n(3n+1)}2$. Denote Jacobi's theta functions by
$$\theta_2(q)=\sum_{\mathbb{Z}} q^{(n+\frac12)^2} \qquad \text{and} \qquad \theta_3(q)=\sum_{\mathbb{Z}} q^{n^2}.$$
The Pentagonal Number Theorem gives $(q;q)_{\infty}=\sum_{\mathbb{Z}} (-1)^n\,q^{\omega(n)}$. By a classical theorem of Jacobi on representations of a number as a sum of two squares, we have
$$U_1(0,q)=\sum_{n\geq1}\frac{q^n}{1+q^{2n}}=\frac{\theta_3(q)^2-1}4.$$
On the other hand, 
$$U_1(1,q)=\sum_{n\geq1}\frac{q^n}{1+q^n+q^{2n}}=\frac{\sum_{\mathbb{Z}}(-1)^n\, n\,q^{\omega(n)}}{\sum_{\mathbb{Z}} (-1)^n\,q^{\omega(n)}}.$$
This follows from differentiation of a formula from Ramanujan's \emph{Lost Notebook}.  However, we may also extract this directly from Jacobi's Triple Product formula (see \cite[p. 35, Entry 19]{Berndt})
$$\prod_{n\geq1} (1-q^{3n})(1-\zeta q^{3n-1})(1-\zeta^{-1}q^{3n-2})=\sum_{\mathbb{Z}} (-1)^nq^{\omega(n)}\zeta^n$$
and then computing the derivative $\frac{d}{d\zeta}$ at $\zeta=1$ so that
$$\prod_{n\geq1}(1-q^n) \cdot \left[\sum_{n\geq1}\frac{q^{3n-2}}{1-q^{3n-2}}-\sum_{n\geq1}\frac{q^{3n-1}}{1-q^{3n-1}}\right]
=\sum_{\mathbb{Z}} (-1)^n \, n \, q^{\omega(n)}.$$
The claim follows after dividing through by $\prod_n(1-q^n)$ and observing that
$$\sum_{n\geq1}\frac{q^{3n-2}}{1-q^{3n-2}}-\sum_{n\geq1}\frac{q^{3n-1}}{1-q^{3n-1}}
=\sum_{n\geq1}\frac{q^n}{1-q^{3n}}-\sum_{n\geq1}\frac{q^{2n}}{1-q^{3n}}=\sum_{n\geq1}\frac{q^n(1-q^n)}{1-q^{3n}}.$$
\smallskip
\noindent
Next, we consider
$$1+2U_1(-1,q)=\sum_{\mathbb{Z}}\frac{q^n}{1+q^{3n}}+\sum_{\mathbb{Z}}\frac{q^{2n}}{1+q^{3n}}.$$
Now by the Ramanujan ${}_1\psi_1$ formula \cite[Eq. (5.2.1)]{Gasper}, we may deduce
$$\sum_{\mathbb{Z}}\frac{t^n}{1-cq^n}=\frac{(q)^2_{\infty}(ct;q)_{\infty}(q/ct;q)_{\infty}}{(c;q)_{\infty}(q/c;q)_{\infty}(t;q)_{\infty}(q/t;q)_{\infty}}.$$
Hence with $q\rightarrow q^3, c=-1$ and $t=q$, we obtain
\begin{align*}
\sum_{\mathbb{Z}}\frac{q^n}{1+q^{3n}}&=\frac{(q^3;q^3)^2_{\infty}(-q;q^3)_{\infty}(-q^2;q^3)_{\infty}}{2\,(-q^3;q^3)^2_{\infty}(q;q^3)_{\infty}(q^2;q^3)_{\infty}} 
=\frac12\, \frac{\theta_3(-q^3)^3}{\theta_3(-q)}.
\end{align*}
And sending $n\rightarrow -n$,
$$\sum_{\mathbb{Z}}\frac{q^{2n}}{1+q^{3n}}=\sum_{\mathbb{Z}}\frac{q^n}{1+q^{3n}}.$$
Hence
$$U_1(-1,q)=\frac12\left(\frac{\theta_3(-q^3)^3}{\theta_3(-q)}-1\right).$$
Alternate formula for $U_1(1,q^4)$.
\begin{align*}
U_1(-1,q)-\sum_{n\geq1}\frac{q^n}{1+(-q)^n+q^{2n}}
&=\sum_{n\geq1} \left(\frac{q^n}{1-q^n+q^{2n}}- \frac{q^n}{1+(-q)^n+q^{2n}}\right) \\
&=\sum_{n\geq1} \left(\frac{q^{2n}}{1-q^{2n}+q^{4n}}- \frac{q^{2n}}{1+q^{2n}+q^{4n}}\right) \\
&=\sum_{n\geq1} \frac{2q^{4n}}{(1+q^{4n})^2-q^{4n}} = \sum_{n\geq1} \frac{2q^{4n}}{1+q^{4n}+q^{8n}} \\
&=2U_1(1,q^4).
\end{align*}
Now by $A113661$ in OEIS~\cite{OEIS},
$$\sum_{\mathbb{Z}} \frac{q^n}{1+(-q)^n+q^{2n}}=\frac16\left(\frac{\theta_3(q)^3}{\theta_3(q^3)}-1\right).$$
Hence
\begin{align*}
U_1(1,q^4)&=\frac12\, U_1(-1,q) - \frac1{12}\left(\frac{\theta_3(q)^3}{\theta_3(q^3)}-1\right) \\
&=\frac14\left(\frac{\theta_3(-q^3)^3}{\theta_3(-q)}-1\right) -\frac1{12}\left(\frac{\theta_3(q)^3}{\theta_3(q^3)}-1\right).
\end{align*}
So
$$U_1(1,q^4)=\frac14\, \frac{\theta_3(-q^3)^3}{\theta_3(-q)} -\frac1{12}\, \frac{\theta_3(q)^3}{\theta_3(q^3)} -\frac16.$$

\begin{lemma} \label{basecases} We have the following three identities:
\begin{align*}
U_1(0,q)&=\frac{\theta_3(q)^2-1}4, \\
U_1(1,q)&= \frac{\theta_2(q)\theta_2(q^3)+\theta_3(q)\theta_3(q^3)-1}6, \\
U_1(-1,q)&=\frac{2\theta_2(q^2)\theta_2(q^6)+2\theta_3(q^2)\theta_3(q^6)+\theta_2(q)\theta_2(q^3)+\theta_3(q)\theta_3(q^3)-3}6.
\end{align*}
\end{lemma}
\begin{proof} The second formula holds due to the classical result \cite{Hirshhorn} that 
$$\sum_{a, b \in\mathbb{Z}}q^{a^2+ab+b^2}=1+6\left(\sum_{n\geq1}\frac{q^{3n-2}}{1-q^{3n-2}}-\sum_{n\geq1}\frac{q^{3n-1}}{1-q^{3n-2}}\right)$$ 
and the equality of the two multisets defined as $\mathcal{A}:=\{a^2+ab+b^2: \, a, b \in\mathbb{Z}\}$ and
$\mathcal{B}:=\{n^2+3m^2: \, n, m\in \mathbb{Z}\} \,\cup \,
\{n^2+3m^2 + n+3m + 1: \, n, m \in \mathbb{Z}\}$ (equivalence of quadratic forms) which lead to
\begin{align*}
\sum_{a,b\in\mathbb{Z}}q^{a^2+ab+b^2}
=\sum_{n\in\mathbb{Z}}q^{(n+\frac12)^2}\sum_{m \in\mathbb{Z}}q^{3(m+\frac12)^2} 
+  \sum_{n\in\mathbb{Z}}q^{n^2}\sum_{m\in\mathbb{Z}}q^{3m^2}.
\end{align*}
The last identity follows from the elementary observation that
$$U_1(-1,q)-U_1(1,q)=\sum_{n\geq1}\left(\frac{q^n(1+q^n)}{1+q^{3n}}-\frac{q^n(1-q^n)}{1-q^{3n}}\right)=2\sum_{n\geq1} \frac{q^{2n}(1-q^{2n})}{1-q^{6n}}. \qedhere$$
\end{proof}

\begin{example} We have that
\begin{align*}
\theta_3(q)\theta_3(q^3)&=2U_1(1,q)+4U_1(1,q^4)+1, \\
U_2(1,q)&=-\,D\left(\log (q)_{\infty}\right) \Bigr|_{ q\rightarrow q^3}   =-\, \frac{\sum_{\mathbb{Z}}(-1)^n\, \omega_n\,q^{3\omega_n}}{\sum_{\mathbb{Z}} (-1)^n\,q^{3\omega_n}}=\frac{1-E_2(q^3)}{24}, \\
U_1(0,q)&= \frac{(q^2)_{\infty}\, \sum_{\mathbb{Z}} (-1)^n \, n\, q^{\binom{2n+1}2}}{(q)_{\infty}(q^4)_{\infty}} \qquad \text{(see Section $5$)}.
\end{align*}
\end{example}

\begin{proposition}
Suppose that
$$f_t(n):=\sum_{k=0}^n \frac{(-1)^k(2n+1)\binom{2n+1-k}k}{2n+1-k}\binom{n-k}t  3^{n-k-t}.$$
Then, the recursive formula $f_t(n+2)-f_t(n+1)+f_t(n)=f_{t-1}(n)$ holds and hence
\begin{align*}
U_{t-1}(1,q) 
&=U_t(1,q)+\frac1{(q^3;q^3)_{\infty}} \left(\sum_{n\geq0}f_t(n+2)\,q^{\binom{n+1}2}-\sum_{n\geq0}f_t(n+1)\,q^{\binom{n+1}2}\right).
\end{align*}
\end{proposition}

\section{Quasimodular structure when $a\in \{-1,0,1\}$}

\noindent
We require some definitions first (see for example \cite{EichlerZagier}).

\begin{definition}A holomorphic function $\phi(\tau,z)$ on $\mathbb{H}\times\mathbb{C}$ is a  {\it Jacobi form for a congruence subgroup $\Gamma\subseteq \rm{SL}_2(\mathbb{Z})$ of weight $k$ and index $m$} if it satisfies the following conditions:

\noindent
(1) For all $\left(\begin{smallmatrix}a&b\\c&d\end{smallmatrix}\right)\in \Gamma,$ we have the modular transformation
\[\phi\left(\frac{a\tau+b}{c\tau+d}, \frac{z}{c\tau+d}\right) \ = \ (c\tau+d)^k\exp\left(2\pi i\cdot \frac{mcz^2}{c\tau+d}\right) \phi(\tau,z).\]

\noindent
(2)  For all integers $a,b$, we have the elliptic transformation
\[
\phi(\tau,z+a\tau+b)  \ = \  \exp\big(-2\pi i m(a^2 \tau+2az)\big) \phi(\tau,z).
\]

\noindent
(3) For each $\left(\begin{smallmatrix}a&b\\c&d\end{smallmatrix}\right)\in \rm{SL}_2(\mathbb{Z})$, we have the Fourier expansion
\[(c\tau+d)^{-k}\exp\left(-2\pi i\cdot \frac{mcz^2}{c\tau+d}\right) \phi\left(\frac{a\tau+b}{c\tau+d}, \frac{z}{c\tau+d}\right)
=\sum_{n\geq 0}\sum_{r^2\leq 4mn} b(n,r)q^nu^r;\]
where $b(n,r)$ are complex numbers and $u:=e^{2\pi i z}$.
\end{definition}

\smallskip
\noindent
We also recall a result on Jacobi forms \cite[Theorem 1.3]{EichlerZagier}.

\begin{theorem} \label{Taylor} Let $\phi$ be a Jacobi form on $\Gamma$ of weight $k$ and index $m$ and $\lambda, \mu$ rational numbers. Then the function 
$f(\tau) = e^{m\lambda^2\tau} \phi(\tau, \lambda\tau+\mu)$ is a modular form (of weight $k$ and on some subgroup of $\Gamma'$ of finite index depending only on 
$\Gamma$ and $\lambda, \mu$).
\end{theorem}

\noindent
Next, we state and prove the main result of this section.

\begin{theorem} \label{QM01-1} Fix $a\in \{-1,0,1\}$. Then, for each non-negative integer $t$, the functions $U_t(a,q)$ are quasimodular forms of mixed weight for a congruent subgroup $\Gamma$.
\end{theorem}
\begin{proof} It's our convention that $U_0(a,q):=1$. Let's expand the generating function
\begin{align*}
\sum_{t\geq0} U_t(a,q)\, x^{2t} & = \prod_{n\ge1} \left(1+\frac{x^2\,q^n}{1+aq^n+q^{2n}}\right) \\
&=\prod_{n\geq1} \frac{1+(x^2+a)q^n+q^{2n}}{1+aq^n+q^{2n}}.
\end{align*}
Choose $\zeta$ such that $\zeta+\zeta^{-1}=x^2+a$ in the Jacobi's Triple Product \cite[Theorem 2.8]{Andrews} to obtain
\begin{align*}
\sum_{t\geq0} U_t(a,q)\, x^{2t}&= \prod_{n\geq1} \frac{(1-q^n)\,(1+\zeta q^n)(1+\zeta^{-1}q^n)}{(1-q^n)(1+aq^n+q^{2n})} \\
&= \frac1{(\sqrt{\zeta}+\frac1{\sqrt{\zeta}}) \, q^{\frac18}}\cdot
\frac{\sum_{m\in\mathbb{Z}+\frac12}  q^{\frac12m^2} \zeta^m}{\prod_{n\geq1}(1-q^n)(1+aq^n+q^{2n})}; 
\end{align*}
where $\vartheta_{\frac12}(q,\zeta):=\sum_{m\in\mathbb{Z}+\frac12}q^{\frac12m^2}\zeta^m$ is a Jacobi form of weight $\frac12$ and index $\frac12$.

\smallskip
\noindent
Since $\sqrt{\zeta}+\frac1{\sqrt{\zeta}}=\sqrt{x^2+a+2}$, we may also write the above in the form
\begin{align*}
\sqrt{a+2}\cdot \sqrt{1+\frac{x^2}{a+2}} \cdot \sum_{t\geq0} U_t(a,q)\, x^{2t}
&= \frac{\sum_{m\in\mathbb{Z}+\frac12}  q^{\frac12m^2} \zeta^m}{q^{\frac18}\prod_{n\geq1}(1-q^n)(1+aq^n+q^{2n})}.
\end{align*}
Now, the right-hand side is a Jacobi form of weight $0$ and index $\frac12$ for $a=0, \pm1$.
While the left-hand side amounts to the convolution
$$\sqrt{a+2}\cdot \sum_{n\geq0}\left(\sum_{k=0}^n \frac1{(a+2)^k}\binom{\frac12}k U_{n-k}(a,q)\right)x^{2n}
= \frac{\sum_{m\in\mathbb{Z}+\frac12}  q^{\frac12m^2} \zeta^m}{q^{\frac18}\prod_{n\geq1} (1-q^n)(1+aq^n+q^{2n})}.$$
Recall the Dedekind eta-function $\eta(q)=q^{\frac1{24}}\prod_{n\geq1}(1-q^n)$.

\bigskip
\noindent
\bf The case $a=1:$ \rm Expand the RHS at $z=\frac16$ or $\zeta=\frac{1+i\sqrt3}2=e^{2\pi i z}$ (so $\zeta^6=1$).
$$ \sqrt{1+\frac{x^2}3} \cdot \sum_{t\geq0} U_t(1,q)\, x^{2t}
= \frac1{\sqrt3} \frac{\sum_{m\in\mathbb{Z}+\frac12}  q^{\frac12m^2} \zeta^m}{\eta(q^3)}.$$
This is equal to the convolution
$$\sum_{n\geq0}\left(\sum_{k=0}^n \frac1{3^k}\binom{\frac12}k U_{n-k}(1,q)\right)x^{2n}
= \frac1{\sqrt3} \frac{\sum_{m\in\mathbb{Z}+\frac12}  q^{\frac12m^2} \zeta^m}{\eta(q^3)}.$$
\bf The case $a=-1:$ \rm Expand the RHS at $z=\frac13$ or $\zeta=\frac{-1+i\sqrt3}2=e^{2\pi i z}$ (so $\zeta^3=1$).
$$\sqrt{1+x^2} \cdot \sum_{t\geq0} U_t(-1,q)\, x^{2t}
= \frac{\eta(q^2)\, \eta(q^3) \sum_{m\in\mathbb{Z}+\frac12}  q^{\frac12m^2} \zeta^m}{ \eta(q)^2 \, \eta(q^6)}.$$
This amounts to the convolution
$$\sum_{n\geq0}\left(\sum_{k=0}^n \binom{\frac12}k U_{n-k}(-1,q)\right)x^{2n}
= \frac{ \eta(q^2)\, \eta(q^3) \sum_{m\in\mathbb{Z}+\frac12}  q^{\frac12m^2} \zeta^m}{ \eta(q)^2 \, \eta(q^6)}.$$
\bf The case $a=0:$ \rm Expand the RHS at $z=\frac14$ or $\zeta=i=e^{2\pi iz}$ (so $\zeta^4=1$).
$$ \sqrt{1+\frac{x^2}2} \cdot \sum_{t\geq0} U_t(0,q)\, x^{2t}
= \frac1{\sqrt2} \frac{ \eta(q^2) \sum_{m\in\mathbb{Z}+\frac12}  q^{\frac12m^2} \zeta^m}{ \eta(q)\, \eta(q^4)}.$$
This is equal to the convolution
$$\sum_{n\geq0}\left(\sum_{k=0}^n \frac1{2^k}\binom{\frac12}k U_{n-k}(0,q)\right)x^{2n}
= \frac1{\sqrt2} \frac{ \eta(q^2) \sum_{m\in\mathbb{Z}+\frac12}  q^{\frac12m^2} \zeta^m}{ \eta(q)\, \eta(q^4)}.$$
By employing Theorem~\ref{Taylor}, we know that the Taylor series coefficients at $z=\frac16, \frac13, \frac14$ of the respective right-hand sides are quasimodular forms of pure weight. On the other hand, the corresponding left-hand side coefficients can recursively determine that the functions $U_n(a,q)$ are themselves quasimodular forms (of mixed weight) once we realize $U_0(a,q)=1$ are modular of weight $0$ and by Lemma~\ref{basecases} each $U_1(a,q)$ is quasimodular of weight $1$.
\end{proof}

\begin{remark} Although we did not explicitly pursue this point in Theorem~\ref{QM01-1}, we believe that each function $U_t(a,q)$ belongs to $\widetilde{M}_{\leq t}(\Gamma_0(24))$ for the congruent subgroup $\Gamma_0(24)$.
\end{remark}

\section{Appendix 1 - WZ's approach}

\noindent
In this section, we opt to verify at least two of the binomial coefficient identities which appeared in the earlier sections of this paper. Our proof is the so-called Wilf-Zeilberger (WZ) method of automated procedure \cite{WZ}, effective for identities of hypergeometic type.

\begin{example} \label{WZ1} We have

$$(2n+1)\sum_{k=0}^n (-1)^{n+k} \frac{\binom{n+k+1}{2k+1}}{n+k+1}\binom{k}{t} 4^{k-t}=\binom{n+t}{2t}.$$

\end{example}

\begin{proof} Define two functions

\begin{align*}
f_1(n,k):&=\frac{(-1)^{n+k}(2n+1)}{n+k+1} \frac{\binom{n+k+1}{2k+1}\binom{k}{t}}{\binom{n+t}{2t}} 4^{k-t}, \qquad \text{and} \\
g_1(n,k):&=f_1(n,k)\cdot \frac{2(n+1)(k-t)(2k+1)}{(2n+1)(n+t+1)(n-k+1)}
\end{align*}
where the second function is generated by \emph{Zeilberger's algorithm}. Then, one checks (using symbolic software!) that
$f_1(n+1,k)-f_1(n,k)=g_1(n,k+1)-g_1(n,k)$. Next, sum both sides of the last equation over all integers $k$. The outcome is the right-hand side vanishes and hence the sum $\sum_kf_1(n,k)$ is a constant (independent of $n$). Keep in mind that actually these sums have ``finite support'', namely the summands are zero outside of a finite interval. To complete the argument, evaluate say at $n=t$ to obtain the value $1$. That means $\sum_kf_1(n,k)=1$ hence the desired claim follows.
\end{proof}
 
\begin{example} \label{WZ2} We have
$$(2n+1)\sum_{k=0}^n (-1)^{n+k} \frac{\binom{n+k+1}{2k+1}}{n+k+1}\binom{k}{2}3^{k-2}
=\begin{cases}
(-1)^{j-1}\frac{j(3j+1)}{2}&\text{if $n=3j$,}\\
0&\text{if $n=3j+1$,}\\
(-1)^{j-1}\frac{j(3j-1)}{2}&\text{if $n=3j-1$.}
\end{cases}$$
\end{example}
\begin{proof} We limit our justification to just one of the cases, say $n\rightarrow 3n$, since the remaining two are worked out analogously. To this end, introduce the discrete functions
\begin{align*}
f_2(n,k):&= \frac{(-1)^{k-1}(6n+1)}{n(3n+1)(3n+k+1)} \binom{3n+k+1}{2k+1}\binom{k}2 3^{k-2}, \qquad \text{and} \\
g_2(n,k):&= f_2(n,k)\cdot R(n,k)
\end{align*}
where $R(n,k)$ is some rational function of $n$ and $k$ which is \emph{too long} to exhibit here but it can be furnished upon request. The next few steps are entirely similar to the above example, hence are omitted.
\end{proof}

\section{Appendix 2 - Riordan's approach}

\noindent
Here, we give a detailed account of some computations made in Section 5. From equation \eqref{coeffCheby}, we have that
\begin{align*}
[x^t]\,to_n\left(\frac{x+a+2}{4}\right)
&=(2n+1)\sum_{k=0}^n (-1)^{n+k} \frac{\binom{n+k+1}{2k+1}}{n+k+1}\binom{k}{t}(a+2)^{k-t}\\
&=(-1)^{n-t}\sum_{k=0}^n \frac{2n+1}{2k+1}\binom{n+k}{2k}\binom{k}{t}(-a-2)^{k-t}.
\end{align*}
By applying Riordan arrays (see \cite{Sprugnoli} for more details) to the last binomial sum, we find that
\begin{align*}
[x^t]\,to_n\left(\frac{x+a+2}{4}\right)
&=(-1)^{n-t}[z^n]\frac{z^t(1+z)}{(1+az+z^2)^{t+1}}.
\end{align*}
If $a=2$ then we immediately obtain
$$[x^t]\,to_n\left(\frac{x+4}{4}\right)
=(-1)^{n-t}[z^n]\frac{z^t}{(1+z)^{2t+1}}=\binom{n+t}{2t}.$$
In a similar way, for $a=-2$,
$$[x^t]\,to_n\left(\frac{x}{4}\right)
=(-1)^{n-t}[z^n]\frac{z^t(1+z)}{(1-z)^{2t+2}}=(-1)^{n-t}\left(\binom{n+t+1}{2t+1}+\binom{n+t}{2t+1}\right).$$
Moreover, for $a=0$,
$$[x^t]\,to_n\left(\frac{x+2}{4}\right)
=(-1)^{n-t}[z^n]\frac{z^t(1+z)}{(1+z^2)^{t+1}}=
\begin{cases}
(-1)^{n+j}\binom{j}{t}&\text{if $n+t=2j$,}\\
(-1)^{n+j}\binom{j}{t}&\text{if $n+t=2j+1$.}
\end{cases}$$
Finally,  if $a=1$ then
\begin{align*}
[x^t]\,to_n\left(\frac{x+3}{4}\right)
&=(-1)^{n-t}[z^n]\frac{z^t(1+z)(1-z)^{t+1}}{(1-z^3)^{t+1}}\\
&=(-1)^{n-t}\left([z^{n-t}]h(z)+[z^{n-t-1}]h(z)\right)
\end{align*}
where $h(z)=(\frac{1-z}{1-z^3})^{t+1}$.

\begin{example} \label{R2} If $a=1$ and $t=2$ then
\begin{align*}
[x^2]\,to_n\left(\frac{x+3}{4}\right)
&=(-1)^{n}\left([z^{n-2}]\left(\frac{1-z}{1-z^3}\right)^{3}+[z^{n-3}]\left(\frac{1-z}{1-z^3}\right)^{3}\right)\\
&=\begin{cases}
(-1)^{j-1}\frac{j(3j+1)}{2}&\text{if $n=3j$,}\\
0&\text{if $n=3j+1$,}\\
(-1)^{j-1}\frac{j(3j-1)}{2}&\text{if $n=3j-1$.}
\end{cases}
\end{align*}
where $h(z)=\left(\frac{1-z}{1-z^3}\right)^{3}=\frac{1}{(1+z+z^2)^3}=\sum_{n=0}^{\infty}a_n z^n$ with
$$a_n=\begin{cases}
j+1&\text{if $n=3j$,}\\
-\frac{3(j+1)(j+2)}{2}&\text{if $n=3j+1$,}\\
\frac{3j(j+1)}{2}&\text{if $n=3j-1$.}
\end{cases}$$
\rm(\it see $A128504$ in OEIS~\cite{OEIS}\rm).
\end{example}

\end{document}